\documentclass{amsart}
\usepackage{graphicx}
\usepackage[mathscr]{eucal}
\usepackage{amssymb}
\newtheorem{thm}{Theorem}

\newtheorem{lem}[thm]{Lemma}
\newtheorem{prop}[thm]{Proposition}
\theoremstyle{definition}

\theoremstyle{remark}

\newtheorem{rems}[thm]{Remarks}

\newcommand{\A}{\mathcal{A}}
\newcommand{\B}{\mathcal{B}}
\newcommand{\C}{\mathcal{C}}

\newcommand{\Z}{\mathcal{Z}}
\newcommand{\Pe}{\mathcal{P}}

\newcommand{\Pm}{\mathrm{Prim}}
\newcommand{\Gl}{\mathrm{Glimm}}
\newcommand{\mset}{\emptyset}

\newcommand{\pr}{{}^{\prime}}
\newcommand{\I}{\mathcal{I}}
\newcommand{\D}{\mathfrak{D}}

\begin{document}
\baselineskip=18pt
\title{Glimm Spaces of Separable $C^*$-algebras}%
\bigskip
\bigskip
\dedicatory{In memory of Ola Bratteli}
\bigskip
\bigskip
\author{Aldo J. Lazar}%
\address{School of Mathematical Sciences\\
         Tel Aviv University\\
         Tel Aviv 69978, Israel}%
\email{aldo@post.tau.ac.il}
\author{Douglas W. B. Somerset}
\address{Institute of Mathematics, University of Aberdeen\\
         King's College, Aberdeen AB24 3UE, Scotland, UK}
\email{dwbsomerset@googlemail.com}

\thanks{The authors wish to thank Professors A. V. Arhangel'ski\u{i} and G Gruenhage from whom they learned about the existence of topological spaces with certain properties discussed in the paper}
\subjclass{46L05}
\keywords{AF $C^*$-algebra, Glimm space}

\date{September 17, 2015}
\begin{abstract}

We characterize the topology of the Glimm space of a separable $C^*$-algebra and extend the main result of \cite{B} to non-unital AF
$C^*$-algebras.

\end{abstract}
\maketitle

\newpage
\section{Introduction and Preliminaries} \label{S:I}

Let $A$ be a $C^*$-algebra with multiplier algebra $M(A)$. The celebrated Dauns-Hofmann theorem states that $Z(M(A))$, the centre of $M(A)$, is
$*$-isomorphic to $C^b(\Pm(A))$, the $C^*$-algebra of continuous, bounded, complex-valued functions on $\Pm(A)$ (the primitive ideal space of
$A$ with the hull-kernel topology). Let $\phi_A : \Pm(A)\to \Gl(A)$ be the complete regularization map for $\Pm(A)$ (where $\Gl(A)$ is the
completely regular topological space with the universal property that any continuous map of $\Pm(A)$ into a completely regular space factors
uniquely through $\phi_A$). Thus $Z(M(A))$ is $*$-isomorphic to $C^b(\Gl(A))$.

If $A$ is a unital $C^*$-algebra then $A=M(A)$ and $\Gl(A)$ is compact (being the continuous image of the compact space $\Pm(A)$) and is
homeomorphic to the maximal ideal space of $Z(A)$; and hence is metrizable if $A$ is also separable. If $A$ is non-unital, however, there may be
little relation between $\Gl(A)$ and $Z(A)$. For example, if $A=C^*(G)$, where $G$ is the continuous Heisenberg group, then $Z(A)=\{0\}$ but
$\Gl(A)$ is homeomorphic to the real line ${\bf R}$, see \cite[section III]{R.-Y.}. This phenomenon, where the Glimm space is larger than the
maximal ideal space of the centre, is common with the group C$^*$-algebras of non-discrete groups as can be seen in \cite[section 5]{K} and
\cite[sections 1, 3, and 5]{Lo}. Another complication that arises in the non-unital context is that $\Gl(A)$ need not be locally compact (see,
for example \cite[III.9.2]{DH}), and hence may not be homeomorphic to the maximal ideal space of any commutative $C^*$-algebra.

If $A$ has a countable approximate identity then $\Pm(A)$ is $\sigma$-compact and hence its continuous image $\Gl(A)$ is also $\sigma$-compact,
and is thus normal and paracompact. Otherwise little is known about the possibilities for $\Gl(A)$, and the main purpose of this paper is to
show that, for $A$ separable, $\Gl(A)$ can be any quotient of a locally compact, second countable, Hausdorff space, and that all possibilities
occur as the Glimm spaces of separable AF $C^*$-algebras (Theorem \ref{T:Eq}). This extends the theorem of O. Bratelli who showed that every
separable, unital, commutative $C^*$-algebra arises as the centre of some separable unital AF $C^*$-algebra \cite{B}. In particular it follows
that ${\bf Q}$ (the space of rational numbers with the usual topology) is not the Glimm space of any separable $C^*$-algebra. On the other hand,
Theorem \ref{T:Eq}, combined with an example due to Arhangel'ski\u{i} and Franklin \cite{AF}, shows that there is a separable AF $C^*$-algebra
$A$ for which $\Gl(A)$ is nowhere locally compact. The basic idea of the construction that we use for a $C^*$-algebra having a prescribed Glimm
space is due to Bratteli \cite[p. 199]{B}; however, the proof that this $C^*$-algebra satisfies the requirements is new even for the case of a
second countable locally compact Hausdorff space that was treated in \cite{B} and \cite{HT}.

In Section \ref{S:Points} we investigate the failure of local compactness of $\Gl(A)$ in the case when $A$ is separable. We show that a point
$x\in \Gl(A)$ fails to have a compact neighbourhood precisely when it fails to have a countable neighbourhood base, and that this can be
characterized in terms of the behaviour of the map $\phi_A$ at $\phi_A^{-1}(x)$ (Theorem \ref{T:Top}). Thus if $A$ is separable, then $\Gl(A)$
decomposes into an open subset $S$, which is both locally compact and metrizable, and a closed subset $\Gl(A)\setminus S$ of exceptional points
for which both local compactness and first countability fail.

In the proof of Theorem \ref{T:Eq} we shall employ Bratteli diagrams (see \cite{Ind} and \cite{LT} for the definition of this notion) of the
kind that we describe below. They consist of a set of pairs, called also vertices, $\mathfrak{D} := \{(k,i) \mid 1\leq i\leq s(k), k\geq 1\}$
where $\{s(k)\}_{k=1}^{\infty}$ is a sequence of natural numbers and a set of edges connecting the elements of the $k$-th row $\{(k,i) \mid
1\leq i\leq s(k)\}$ with elements of the $(k+1)$-th row. If $(k,i)$ and $(k+1,j)$ are connected by an edge, one says that $(k,i)$ is an
(immediate) ancestor of $(k+1,j)$ and the latter is a descendant of the former. Each element in $\mathfrak{D}$ has at least one descendant. In
general edges can have a multiplicity but in this paper every edge has multiplicity one. The dimension of each pair that has no ancestors is
one; the dimension attached to any other pair is equal to the sum of the dimensions of all its ancestors in the previous row. More generally, we
shall say that the vertex $v_1$ is an ancestor of the vertex $v_p$ if there are vertices $\{v_t\}_{t=1}^p$ such that $v_i$ is an immediate
ancestor of $v_{i+1}$, $1\leq i<p$. Of course, in this situation too, $v_p$ is called a descendant of $v_1$. We shall say that a sequence
$\{v_p\}_{p=1}$ of vertices of the diagram, finite or infinite, is connected if $v_p$ is an ancestor of $v_{p+1}$, $p\geq 1$. A connected
sequence is called complete if $v_{p+1}$ is on the row immediately following that of $v_p$, $p\geq 1$. It is well known that every Bratteli
diagram defines an AF $C^*$-algebra. Moreover, as shown in \cite{Ind} (see also \cite{LT} for the non-unital case), there is a one-to-one
correspondence between the closed ideals of this AF algebra and subdiagrams of a certain kind which we shall call ideal subdiagrams. This
correspondence respects the order of the lattice of closed ideals and that of the subdiagrams. By a primitive ideal subdiagram we shall mean an
ideal subdiagram that corresponds to a primitive ideal. Let now $\{v_p\}_{p=1}^{\infty}$ be a complete connected sequence in $\mathfrak{D}$ with
$v_1$ on the first row of $\mathfrak{D}$. Set
$$
  \mathfrak{E} := \{v\in \mathfrak{D} \mid v\ \text{is a descendant of $v_p$ for some}\ p\geq 1\}.
$$
Then $\mathfrak{E}$ is the smallest subset of $\mathfrak{D}$ having the properties that $\{v_p\}_{p=1}^{\infty}\subseteq \mathfrak{E}$ and
$\mathfrak{D}\setminus \mathfrak{E}$ is a primitive ideal subset.

The space $\Gl(A)$ associated with a $C^*$-algebra $A$ is endowed with the weakest topology for which the elements of $C^b(\Pm(A))$ when
considered as functions on $\Gl(A)$ are continuous. For a separable $A$ this is the same as the quotient topology defined by $\phi_A$ as follows
from \cite[Theorem 2.6]{L}.

For the proof of our main result we shall need a variant of lemma from \cite{B}.

\begin{lem}[Lemma 4 of \cite{B}] \label{L:norm}

   Suppose $\A_1\subseteq \A_2$ are finite dimensional $C^*$-algebras with the unit of $\A_2$ contained in $\A_1$ and $\{e_k^j\}_{k=1}^{p_j}$, $j = 1,2$,
   are mutually orthogonal projections in
   $\Z(\A_j)$. Let $a_j = \Sigma_{k=1}^{p_j} \alpha_k^je_k^j$, $j = 1, 2$. Then $\|a_1 - a_2\| = \sup\{|\alpha_k^1 - \alpha_l^2| \mid
   e_k^1e_l^2\neq 0\}$.

\end{lem}

\begin{proof}

   It is assumed everywhere in \cite{B} that an embedding of a finite dimensional $C^*$-algebra into another preserves the units but this is not
   really needed for the proof of \cite[Lemma 4]{B}; the proof there remains valid without this assumption.

\end{proof}

\section{Glimm spaces of AF algebras} \label{S:Glimm}

In this section we shall establish a characterization of the Glimm space of a separable $C^*$-algebra.

\begin{thm} \label{T:Eq}

   The following properties are equivalent for a Hausdorff topological space $X$:
   \begin{itemize}
      \item[(i)] There is an AF algebra $\A$ such that $X$ is homeomorphic to $\Gl(\A)$;
      \item[(ii)] There is a separable $C^*$ algebra $\A$ such that $X$ is homeomorphic to $\Gl(\A)$;
      \item[(iii)] $X$ is homeomorphic to the quotient of a locally compact Hausdorff second countable space;
      \item[(iv)] $X = \cup_{n=1} X_n$ where $\{X_n\}_{n=1}$ is an increasing sequence of compact metrizable subspaces such that a subset
      $F$ of $X$ is closed if and only if $F\cap X_n$ is closed for each $n$.
   \end{itemize}

\end{thm}

\begin{rems}

   \begin{enumerate}

      \item A Hausdorff space satisfying (iii) is paracompact by \cite[Theorem 1]{M}. A space $X$ satisfying (iv) is (homeomorphic to)
      the quotient of the disjoint union of $\{X_n\}$, a locally compact Hausdorff second countable space, hence it is also paracompact.
      \item When $X$ is compact, the implication $(iv)\Rightarrow (i)$ is an immediate consequence of the main result of \cite{B} and the
      Banach-Stone theorem.

   \end{enumerate}

\end{rems}

\begin{proof}

   $(i)\Rightarrow (ii)$. This is obvious.

   $(ii)\Rightarrow (iii)$. Let $\A$ be a separable $C^*$-algebra. By \cite[Proposition 3.2]{AB} and \cite[Proposition 2.4]{L} $\Gl(\A)$ is homeomorphic
   to a quotient space of the space of all proper primal ideals of $\A$ endowed with the $\tau_s$ topology. By \cite[Proposition 4.1 and p.525]{A}
   the latter space is locally compact Hausdorff and second countable.

   $(iii)\Rightarrow (iv)$. Suppose $Y$ is a locally compact Hausdorff space with a countable basis and $\varphi : Y\to X$ is a quotient map.
   Let $\{Y_n\}$ be a sequence of compact subsets of $Y$ such that $Y_n\subset \mathrm{int}(Y_{n+1})$, $n\geq 1$, and $\cup Y_n = Y$. Denote $X_n
   := \varphi(Y_n)$. Then $C(X_n)$ can be isomorphically embedded into $C(Y_n)$; hence $C(X_n)$ is separable and $X_n$ is compact metrizable. Now
   suppose that $F\subset X$ satisfies: each $F\cap X_n$ is closed. We want to show that $\varphi^{-1}(F) = \cup \varphi^{-1}(F\cap X_n)$ is
   closed in the metrizable space $Y$. Let $\{y_m\}$ be a sequence in $\varphi^{-1}(F)$ that converges to $y\in Y$. Then $y\in
   \mathrm{int}(Y_p)$ for some $p\geq 2$; hence $\{y_m\}_{m\geq l}\subset \mathrm{int}(Y_p)$ for some $l$. We have $\{\varphi(y_m)\}_{m\geq
   l}\subset F\cap X_p$ and $\lim \varphi(y_m) = \varphi(y)$. Thus $\varphi(y)\in F\cap X_p$ and $y\in \varphi^{-1}(F\cap X_p)\subset
   \varphi^{-1}(F)$. We have obtained that $\varphi^{-1}(F)$ is closed and it follows that $F$ is closed in $X$.

   $(iv)\Rightarrow (i)$. We begin by choosing a sequence $\{f_m\}$ in $C^b(X)$ such that the restrictions $\{f_m\mid_{X_n}\}_{m=1}$ are dense in
   $C(X_n)$ for every $n$. This can be done by choosing for each $n$ a sequence of functions dense in $C(X_n)$ and extending them to
   all of $X$ without increasing their norms. Renumbering the union of these sequences will yield the sequence $\{f_m\}$.

   Set $X_0 := \mset$. For each $m\geq 1$ and $k\geq 1$ we decompose the sets $X_n\setminus X_{n-1}$, $n\geq 1$, into finitely many mutually
   disjoint subsets such that the oscillation
   of $f_m$ on each of these subsets is less than $2^{-k}$. The characteristic functions of these subsets of $X_n\setminus X_{n-1}$
   generate a commutative AF algebra $\B_n$ of scalar functions on $X$ that vanish off $X_n\setminus X_{n-1}$. Clearly each $f_m\mid_{X_n}$
   belongs to $\C_n := \oplus_{i=1}^n \B_i$; hence $C(X_n)$ is a subalgebra of $\C_n$ with the same unit;
   here we look at the elements of $C(X_n)$ as functions defined on all of $X$ that vanish on $X\setminus X_n$ and similarly, the elements of
   $\C_n$ are considered
   as functions on $X$ that vanish off $X_n$. The restriction of the functions in $\C_{n+1} = \C_n\oplus \B_{n+1}$ to $X_n$ is the natural projection
   of $\C_{n+1}$ onto $\C_n$. Of course, it maps $C(X_{n+1})$ onto $C(X_n)$. Denote by $Y_n$
   the maximal ideal space of $\C_n$, $n\geq 1$ and set $Y_0 := \mset$; we shall identify sometimes $\C_n$ with $C(Y_n)$ and consider $Y_n$ as an open and
   closed subspace of $Y_{n+1}$. Due to the fact that $C(X_n)$ is a $C^*$-subalgebra of $C(Y_n)$ that contains the unit of $C(Y_n)$ there
   exists an obvious quotient map $q_n$ of $Y_n$ onto $X_n$; for $x,y\in Y_n$ we have $q_n(x) = q_n(y)$ if and only if $f(x) = f(y)$ for every $f\in C(X_n)$.
   Clearly, $q_{n+1}\mid Y_n = q_n$ and $q_{n+1}^{-1}(X_{n+1}\setminus X_n)\subseteq Y_{n+1}\setminus Y_n$ for every
   $n$. Set $Y := \cup_{n=1} Y_n$ and endow $Y$ with the weak topology determined by the subspaces $\{Y_n\setminus Y_{n-1}\}_{n=1}$. By defining
   $q(s) := q_n(s)$ if $s\in Y_n$ we get a well defined map of $Y$ onto $X$; it is a quotient map since a subset of $X$ is closed (open) if and
   only if its intersections with every $X_n$ are closed (respectively, relatively open).

   Let now $\{\B_{k,n}\}_{k=n}^{\infty}$ be an increasing sequence of finite dimensional $C^*$-subalgebras of $\B_n$ such that the unit of $\B_n$
   belongs to $\B_{n,n}$ and $\B_n = \overline{\cup_{k=n}
   \B_{k,n}}$. Set $\C_{k,n} := \oplus_{p=1}^k \B_{k,p}$ for $1\leq k\leq n$ and $\C_{k,n} := \oplus_{p=1}^n \B_{k,p}$ for $k > n$. Then
   $\{\C_{k,n}\}_{k=1}$  is an increasing sequence of finite dimensional $C^*$-algebras, the unit of $C_n$ belongs to $C_{n,n}$, and
   $\C_n = \overline{\cup_{k=1} \C_{k,n}}$. The minimal projections $e_k^j$ of $\B_{k,n}$, $k\geq n$, yield a decomposition of $Y_n\setminus Y_{n-1}$
   into mutually disjoint
   open and closed subsets $\{E^j_k\}$, $r(k,n-1) < j\leq r(k,n)$, where $\{r(k,p)\}_{p=1}^n$ is a strictly increasing sequence of natural
   numbers and $r(k,0) = 0$. The upper
   indexing of these subsets is done as now indicated. We index the sets $\{E^j_n\}$, $n\in \mathbb{N}$, in an arbitrary manner starting with $j = r(n,n-1) + 1$.
   For the $(k + 1)$-th
   collection, we begin by indexing all the subsets of $E^{r(k,n-1)+1}_k$ in an arbitrary order beginning with $j = r(k+1,n-1) + 1$ then all the subsets
   of $E^{r(k,n-1)+2}_k$ and so forth. Of course these subsets provide a Bratteli diagram for the commutative AF algebra $\C_n$. Its set of vertices
   is
   \begin{multline}
    \mathfrak{D}_n := \cup_{k=1}^n \{(k,j)\mid j = 1, \ldots r(k,1), r(k,1)+1, \ldots r(k,2), \ldots r(k,k)\}\\
                                           \bigcup \cup_{k=n+1}^{\infty} \{(k,j)\mid j=1, \ldots r(k,1), \ldots r(k,n)\}.
   \end{multline}
   The pair $(k,j)$, $r(k,p-1) < j\leq r(k,p)$, is linked by an edge to the pair $(k+1,h)$, $r(k+1,p-1) < h\leq r(k+1,p)$, if and only if $E_{k+1}^h$
   is a subset of $E_k^j$. There is a one to one correspondence between the points $y$ of $Y_n$ and complete connected sequences
   $\{(k,j_k(y))\}_{k=n}^{\infty}$ in $\mathfrak{D}_n$ given by $\{y\} = \cap_k E_k^{j_k(y)}$. Observe that the map $y\to \{(k,j_k(y))\}$ is the
   restriction at $Y_n$ of the analogous map from $Y_{n+1}$ to the set of complete connected sequences in $\mathfrak{D}_{n+1}$.

   Now we proceed to construct a Bratteli diagram of an AF algebra $\A$. It is based upon the set
   \[
    \mathfrak{D} := \cup_{k=1}^{\infty} \{(k,j)\mid j = 1, \ldots r(k,1), r(k,1)+1, \ldots r(k,2), \ldots r(k,k)\}.
   \]
   There is an obvious one to one correspondence between the row of order $k$ of $\mathfrak{D}$ and the decomposition of $Y_k$ into the subsets
   $\{E_k^j\}$, $0 < j\leq r(k,k)$. We shall arrange the pairs in every row from left to right in increasing order of the second coordinate.
   The pair $(k+1,h)$ is a descendant of $(k,j)$ if and only if $E_{k+1}^h\subset E_k^i$ for some $i\leq j$ and there are $y_1\in E_k^j$ and
   $y_2\in E_{k+1}^h$ such that $q(y_1) = q(y_2)$. In particular, this occurs if $E_{k+1}^h\subset E_k^j$ that is: all the edges from every
   $\mathfrak{D}_n$ are present in $\mathfrak{D}$ too. Each row of $\D$ is finite so every pair has only finitely many ancestors and only finitely many
   descendants in the following row. Observe that the points $(k,j)$ with $r(k,k-1)) < j\leq r(k,k)$, $k\geq 1$, and only these points, have no ancestor
   in $\D$. The set $\D$ with the edges described above is the diagram $\mathfrak{G}$ of an AF algebra which we denote $\A$. Associated with this diagram
   there is an increasing sequence of finite dimensional $C^*$-subalgebras $\{\A_k\}$ where $\A_k = \oplus_{j=1}^{r(k,k)} \A_{k,j}$, $k = 1,2,
   \ldots$. We denote the unit of $\A_{k,j}$ by $f_k^j$. Let $\phi_{\A} : \Pm(\A)\to \Gl(\A)$ be the
   complete regularization map of $\A$. Now we proceed to prove that $X$ is a homeomorphic to $\Gl(\A)$.

   For every natural number $n$ the subdiagram $\mathfrak{G}_n$ of $\mathfrak{G}$ whose set of pairs is $\mathfrak{D}_n$ defines, by
   \cite[Theorem 3.3]{Ind} and \cite[Proposition 2.14]{LT}, an ideal $\I_n$ of $\A$. From $\D_n\subset \D_{n+1}$ we get $\I_n\subset \I_{n+1}$;
   hence $\{\Pm(\I_n)\}_{n=1}$ is an increasing sequence of open subsets of $\Pm(\A)$. Moreover, $\Pm(\A) = \cup_{n=1} \Pm(\I_n)$. Indeed,
   if $\Pe\in \Pm(\A)$ then in
   $\D\setminus \D^{\Pe}$ there is a pair without ancestors. This pair has a descendant in the next row which is in $\D\setminus \D^{\Pe}$ too. By
   induction we shall find in $\D\setminus \D^{\Pe}$ a complete connected sequence whose first pair has no ancestors. If this pair is $(k,j)$
   with $r(k,k-1) < j\leq r(k,k)$ then all its descendants are in $\D_k$; thus we found that $\I_k$ is not contained in $\Pe$ thus $\Pe\in \Pm(\I_k)$.
   We claim that $\Pm(\I_n)$ is a compact subset of $\Pm(\A)$. To see this we examine the projection $f_n := \oplus_{j=1}^{r(n,n)} f_n^j$ that
   belongs to $\I_n$. If $\Pe\in \Pm(\I_n)$ then $\D_n$ is not contained in $\D^{\Pe}$; hence in $\D\setminus \D^{\Pe}$ there exists a connected
   sequence that is eventually in $\D_n$ and we get $\|f_n + \Pe\| = 1$. If $\Pe\notin \Pm(\I_n)$ then $\D_n\subset \D^{\Pe}$ and we get
   $\|f_n + \Pe\| = 0$. Thus $\Pm(\I_n) = \{\Pe\in \Pm(\A) \mid \|f_n + \Pe\|\geq 1\}$ and this set is compact by \cite[Proposition 3.3.7]{D}.
   We want to show now that the topology of $\Gl(\A)$ is determined by the sequence of compact sets $\{\phi_{\A}(\Pm(\I_n))\}_{n=1}$, i.e. a
   subset $U$ of $\Gl(\A)$ is open if (and only if) $U\cap \phi_{\A}(\Pm(\I_n))$ is relatively open in $\phi_{\A}(\Pm(\I_n))$ for every $n$.
   Thus suppose that $U$ satisfies this condition and set $V := \phi_{\A}^{-1}(U)$. Denote by $\phi_n$ the restriction of $\phi_{\A}$ to
   $\Pm(\I_n)$. Then $V\cap \Pm(\I_n) = \phi_n^{-1}(U\cap \phi_{\A}(\Pm(\I_n))$ is a relatively open subset of $\Pm(\I_n)$; hence open in
   $\Pm(\A)$. From $V = \cup_{n=1} (V\cap \Pm(\I_n))$ it follows that $V$ is open in $\Pm(\A)$ and we conclude that $U$ is open in $\Gl(\A)$.

   Now we shall define a continuous map $\varphi : Y\to \Pm(\A)$ as follows: for $y\in Y_n\setminus Y_{n-1}$, $\varphi(y)$ is the largest
   primitive ideal $\Pe$ of $\mathcal{A}$ such that $\D\setminus \D^{\Pe}$ contains the complete connected sequence
   $\{(k,j_k(y))\}_{k=n}^{\infty}$. In order
   to establish the continuity of $\varphi$ let $O$ be an open subset of $\Pm(\A)$ and $\I$ the ideal of $\A$ such that $\Pm(\I) = O$. Then
   \[
     \{y\in Y \mid \varphi(y)\in O\} = \{y\in Y \mid (\D\setminus \D^{\varphi(y)})\cap \D^{\I}\neq \mset\}.
   \]
   By the definition of $\varphi(y)$, we have $(\D\setminus \D^{\varphi(y)})\cap \D^{\I}\neq \mset$ if and only if there is a pair
   $(k_0,j_{k_0}(y))\in \D^{\I}$. Then $y\in E_{k_0}^{j_{k_0}(y)}\subset Y_n\setminus Y_{n-1}$ and if $z\in E_{k_0}^{j_{k_o}(y)}$ we have
   $(k_0,j_{k_0}(y)) = (k_0,j_{k_0}(z))$; hence $\varphi(z)$ does not contain the ideal $\I$. Thus
   $\varphi(E_{k_0}^{j_{k_0}(y)})\subset O$ and $\{y\in Y \mid \varphi(y)\in O\}$ is open. The map
   $\varphi$ drops to a continuous map $\psi : X\to \Gl(\A)$; that is, there is a continuous map $\psi$ from $X$ to $\Gl(\A)$
   such that
   \begin{equation} \label{E:1}
      \psi\circ q = \phi_{\A}\circ \varphi.
   \end{equation}
   Indeed, let $y_1,y_2\in Y_n\subset Y$ for some $n$ with $q(y_1) = q(y_2)$. We have
   $\{y_i\} = \cap_{k=n} E_k^{j_k(y_i)}$, $i = 1,2$. Suppose that $j_n(y_1)\leq j_n(y_2)$. Then, in $\D$, each pair $(k,j_k(y_2))$ is connected
   by an edge to the pair $(k+1,j_{k+1}(y_1))$, $k\geq n$. It follows that $\varphi(y_1)\subset \varphi(y_2)$. Thus $\phi_{\A}(\varphi(y_1)) =
   \phi_{\A}(\varphi(y_2))$ and the existence of $\psi$ is proven. The continuity of $\psi$ is an immediate consequence of the
   continuity of the other three maps involved in the equality (\ref{E:1}) and of the fact that $q$ is a quotient map. Clearly
   $\varphi(Y_n)\subset \Pm(\I_n)$ so $\psi(X_n)\subset \phi_{\A}(\Pm(\I_n))$. We intend to prove that $\psi$ is a homeomorphism of $X$ onto
   $\Gl(\A)$.

   The next step will be to show that $\psi(X_n) = \phi_{\A}(\Pm(\I_n))$ for every $n$. Of course, from this will follow that $\psi$ maps $X$
   onto $\Gl(\A)$. So let $\Pe\in \Pm(\I_n)$. There is a complete connected
   sequence $\{(k,j_k)\}_{k=n}^{\infty}$ in $\D\setminus \D^{\Pe}$ with $1\leq j_n\leq r(n,n)$. There exists $s_n$, $1\leq s_n\leq j_n$, such
   that $E_{n+1}^{j_{n+1}}\subset E_n^{s_n}$. Similarly, there exists $s_{n+1}$, $1\leq s_{n+1}\leq s_n$, such that $E_{n+2}^{j_{n+2}}\subset
   E_n^{s_{n+1}}$. Continuing in this way we get a non increasing sequence of natural numbers $\{s_m\}_{m=n}$ that eventually becomes
   constant. Thus we infer that there is $t_n$, $1\leq t_n\leq j_n$, such that all but finitely many of the sets $\{E_k^{j_k}\}_{k=n}$ are contained
   in $E_n^{t_n}$. Arguing with the subsets of $E_n^{t_n}$ corresponding to pairs on the $n+1$ row of $\D$ we find that there is $t_{n+1}$,
   $1\leq t_{n+1}\leq r(n+1,t_n)$ such that $E_{n+1}^{t_{n+1}}$ contains all but finitely many of the sets $\{E_k^{j_k}\}_{k=n}$. Inductively
   one finds a sequence $\{t_l\}_{l=n}$ such that $\{E_l^{t_l}\}_{l=n}$ is decreasing and each $E_l^{t_l}$ contains all but finitely many of
   the sets $\{E_k^{j_k}\}_{k=n}$. Obviously the complete connected sequence $\{(l,t_l)\}_{l=n}^{\infty}$ is contained in $\D\setminus
   \D^{\Pe}$. If we denote by $y$ the unique point in $\cap_{l=n} E_l^{t_l}$ then $y\in Y_n$ and $\varphi(y)\supset \Pe$. It follows that
   $q(y)\in X_n$ and $\psi(q(y)) = \phi_{\A}(\varphi(y)) = \phi_{\A}(\Pe)$.

   Keeping the above notation, we remark that if we choose $y_k\in E_k^{j_k}$, $k = n,\newline n+1,\ldots$, then the sequence $\{y_k\}$ converges to $y$.
   For future use we note that given a primitive ideal $\Pe$ of $\A$
   there is a complete connected sequence $\{(k,t_k)\}_{k=n}^{\infty}$ in $\D\setminus \D^{\Pe}$ such that $\{E_k^{t_k}\}_{k=n}^{\infty}$ is a
   decreasing sequence. Moreover, we claim that if $\{(k,t_k)\}_{k=n}^{\infty}$ and $\{(k,s_k)\}_{k=m}^{\infty}$ are two such sequences in $\D\setminus
   \D^{\Pe}$ with $\{y\} = \cap_{k=n}^{\infty} E_k^{t_k}$ and $\{z\} = \cap_{k=m}^{\infty} E_k^{s_k}$ then $q(y) = q(z)$. Indeed, suppose $m\leq
   n$. Since $\Pe$ is a primitive ideal one can find inductively a connected sequence $\{(h_k,i_k)\}_{k=n}$ in $\D\setminus \D^{\Pe}$ such that
   $(h_k,i_k)$ is a descendant of $(k,t_k)$ and $(k,s_k)$, $k\geq n$. Thus there exist $u_k\in E_k^{t_k}$, $v_k\in E_k^{s_k}$, and
   $\{u_k^{\pr},v_k^{\pr}\}\subset E_{h_k}^{i_k}$, $k\geq n$, such that $q(u_k) = q(u_k^{\pr})$ and $q(v_k) = q(v_k^{\pr})$.
   The sequence $\{u_k\}$ converges to $y$, $\{v_k\}$ converges to $z$ while by the above argument $\{u_k^{\pr}\}$ and $\{v_k^{\pr}\}$ both
   converge to some $w\in Y$. We infer that $q(y) = q(w) = q(z)$.

   Now we shall prove that $\psi$ is one to one and this will be a consequence of the following claim: given $g\in C^b(X)$ there exists $\tilde{g}\in C^b(\Pm(\A))$
   such that
   \begin{equation} \label{E:2}
      \tilde{g}\circ \varphi = g\circ q.
   \end{equation}
   Indeed, assuming this claim to be valid, let $x_1,x_2\in X$ with $x_1\neq x_2$ and choose
   $g\in C^b(X)$ such that $g(x_1)\neq g(x_2)$. Choose also $y_1,y_2\in Y$ such that $q(y_1) = x_1$, $q(y_2) = x_2$. Then
   $\tilde{g}(\varphi(y_1)) = g(x_1)\neq g(x_2) = \tilde{g}(\varphi(y_2))$, hence $\phi_{\A}(\varphi(y_1))\neq \phi_{\A}(\varphi(y_2))$. From
   (\ref{E:1}) we obtain $\psi(x_1) = \psi(q(y_1))\neq \psi(q(y_2)) = \psi(x_2)$. We proceed now to prove the stated claim so we let $g\in
   C^b(X)$. For $\Pe\in \Pm(\A)$ choose a complete connected sequence $\{(k,l_k)\}_{k=n}^{\infty}$ in $\D\setminus \D^{\Pe}$ such that
   $\{E_k^{\l_k}\}$ is a decreasing sequence. With $\{y\} = \cap_{k=n} E_k^{l_k}$ set $\tilde{g}(\Pe) := g(q(y))$. We infer from the remarks above
   that $\tilde{g}$ is well defined. Obviously $\tilde{g}$ is bounded and satisfies (\ref{E:2}). It remains to prove the continuity of
   $\tilde{g}$.

   For $g\in C^b(X)$ we choose a sequence $\{g_k\}_{k=1}^{\infty}$, $g_k =\Sigma_{j=1}^{r(k,k)} \alpha_k^je_k^j$, that converges uniformly on each $Y_n$ to
   $(g\circ q)|_{Y_n}$ and such that $|\alpha_k^j|\leq \|g\|$, $k = 1,2,\ldots $, $1\leq j\leq r(k,k)$. Denote $\eta_k^n := \sup\{y\in Y_n
   \mid |g(q(y)) - g_k(y)|\}$. We consider now the sequence $a_k := \Sigma_{j=1}^{r(k,k)} \alpha_k^jf_k^j\in \A_k$, $k= 1,2,\ldots$. We shall
   show that $\{a_k\}$ converges strictly to an element of the center of the multiplier algebra of $\A$. To this end let $a\in \A_{m,i}$, $1\leq
   i\leq r(m,m)$, and suppose $m<k<l$. Then
   \begin{equation} \label{E:3}
      \|a_ka - a_la\| = \|aa_k - aa_l\| = \|(\sum_j\pr \alpha_k^jf_k^jf_m^i - \sum_j\pr \pr\alpha_l^jf_l^jf_m^i)a\|
   \end{equation}
   where $\sum\pr$ is taken for those indices $j$ such that $f_k^jf_m^i\neq 0$ and $\sum\pr \pr$ is taken for those indices $j$ such that
   $f_l^jf_m^i\neq 0$. Now remark that $f_k^jf_m^i = f_m^if_k^jf_m^i$ and $f_l^jf_m^i = f_m^if_l^jf_m^i$ since $m<k<l$; hence the two sums in (\ref{E:3}) are in
   the unital hereditary $C^*$-algebra $f_m^i\mathcal{A}f_m^i$. They
   actually belong to the centers of $f_m^i\A_kf_m^i$ and $f_m^i\A_lf_m^i$ respectively so we can apply Lemma \ref{L:norm} to estimate the norm
   in (\ref{E:3}). Also remark that if $f_k^{j\pr}f_m^i\neq 0\neq f_l^{j\pr \pr}f_m^i$ and $(l,j\pr \pr)$ is a descendant of
   $(k,j\pr)$ then there are $y\pr\in E_k^{j\pr}$ and $y\pr \pr\in E_l^{j\pr \pr}$ such that $q(y\pr) = q(y\pr \pr)$. Thus
   \begin{equation} \label{E:4}
      |\alpha_k^{j\pr} - \alpha_l^{j\pr \pr}|\leq |\alpha_k^{j\pr} - g(q(y\pr))| + |g(q(y\pr \pr)) - \alpha_l^{j\pr \pr}|\leq
      \eta_k^m + \eta_l^m.
   \end{equation}
   From (\ref{E:3}) and (\ref{E:4}) we obtain
   \begin{equation} \label{E:5}
      \|a_ka - a_la\| = \|aa_k - aa_l\| = \max\{|\alpha_k^{j\pr} - \alpha_l^{j\pr \pr}| \mid f_k^{j\pr}f_l^{j\pr \pr}f_m^i\neq
      0\}\leq \eta_k^m + \eta_l^m.
   \end{equation}
   Thus $\{a_ka\}$ and $\{aa_k\}$ are Cauchy sequences. Of course, this is true if $a$ is any element of $\A_m$. Since the sequence $\{a_k\}$ is
   bounded we can conclude that $\{a_ka\}$ and $\{aa_k\}$ are Cauchy sequences for any $a\in \A$; thus $\{a_k\}$ strictly converges to some
   multiplier $b$. If $a\in \A_m$ and $k > m$ then $aa_k = a_ka$, hence $ab = ba$. It follows that $b$ is in the center of the multiplier
   algebra; hence there is $h\in C^b(\Pm(\A))$ such that $ab + \Pe = h(\Pe)(a + \Pe)$ for every $a\in \A$ and $\Pe\in \Pm(\A)$. We are going to
   show that $h = \tilde{g}$.

   Let $\Pe$ be a primitive ideal of $\A$, say $\Pe\in \Pm(\I_n)$ for some $n$, and suppose $(m,i)\in \D\setminus \D^{\Pe}$ is the leftmost pair of $\D\setminus \D^{\Pe}$ in the
   $m$-th row. Then in every subsequent row of $\D\setminus \D^{\Pe}$ the descendants of $(m,i)$ are situated to the left of every other pair.
   Set $i_m := i$ and for every $k > m$ let $(k,i_k)$ be the leftmost pair of the $k$-th row in $\D\setminus \D^{\Pe}$. The sequence
   $\{E_k^{i_k}\}_{k=m}^{\infty}$ is decreasing and we denote by $y\in Y_n$ the unique point in all these sets. Let $\varepsilon > 0$ be arbitrary and suppose $s\geq m$ is such
   that $|g(q(z)) - g(q(y))| < \varepsilon$ for every $z\in E_s^{i_s}$. An easy induction shows that if $(k,j)\in \D\setminus \D^{\Pe}$, $k\geq
   s$, and $f_k^jf_s^{i_s}\neq 0$ then $E_k^j\subseteq E_s^{i_s}$. Let $k\geq s$ be sufficiently large so $\eta_k^n < \varepsilon$ and $\|bf_s^{i_s} - a_kf_s^{i_s}\| < \varepsilon$.
   Then for $z\in E_k^j\subset Y_n$ we have
   \begin{equation} \label{E:6}
      |g(q(z)) - \alpha_k^j|\leq \eta_k^n < \varepsilon.
   \end{equation}
   Thus, if $(k,j)\in \D\setminus \D^{\Pe}$ and $f_k^jf_s^{i_s}\neq 0$, by taking $z\in E_k^j$ we get
   \begin{equation} \label{E:7}
      |g(q(y)) - \alpha_k^j|\leq |g(q(y)) - g(q(z))| + |g(q(z)) - \alpha_k^j| < 2\varepsilon.
   \end{equation}
   We have
   \begin{multline} \label{E:8}
      \|h(\Pe)(f_s^{i_s} + \Pe) - \tilde{g}(\Pe)(f_s^{i_s} + \Pe)\| = \|(bf_s^{i_s} + \Pe) - g(q(y))(f_s^{i_s} + \Pe)\|\leq \\
          \|(bf_s^{i_s}+\Pe) - \sum\pr \alpha_k^j(f_k^jf_s^{i_s}+\Pe)\| + \|\sum\pr \alpha_k^j(f_k^jf_s^{i_s}+\Pe) - g(q((y))(f_s^{i_s}+\Pe)\|
   \end{multline}
   where the sum is taken for those indices $j$, $1\leq j\leq r(k,k)$, such that $(k,j)\in \D\setminus \D^{\Pe}$ and $f_k^jf_s^{i_s}\neq 0$. Now
   \begin{multline} \label{E:9}
      \|(bf_s^{i_s}+\Pe) - \sum\pr \alpha_k^j(f_k^jf_s^{i_s} + \Pe)\| = \|(bf_s^{i_s}+\Pe) - \sum_{j=1}^{r(k,k)}
      \alpha_k^j(f_k^jf_s^{i_s}+\Pe)\|\leq \\
      \|bf_s^{i_s} - a_kf_s^{i_s}\| < \varepsilon,
   \end{multline}
   \begin{equation} \label{E:10}
      f_s^{i_s} + \Pe = \sum\pr f_k^jf_S^{i_s} + \Pe,
   \end{equation}
   and
   \begin{multline} \label{E:11}
      \|\sum\pr \alpha_k^j(f_k^jf_s^{i_s} + \Pe) - g(q(y))(f_s^{i_s} + \Pe)\| = \|\sum\pr (\alpha_k^j - g(q(y)))(f_k^jf_s^{i_s} + \Pe)\| = \\
      max|\alpha_k^j - g(q(y))| < 2\varepsilon
   \end{multline}
   by (\ref{E:10}) and (\ref{E:7}). From (\ref{E:8}), (\ref{E:9}), and (\ref{E:11}) we obtain
   \[
    \|h(\Pe)(f_s^{i_s} + \Pe) - \tilde{g}(\Pe)(f_s^{i_s} + \Pe)\| < \varepsilon + 2\varepsilon = 3\varepsilon.
   \]
   Since $\|f_s^{i_s} + \Pe\| = 1$ we finally get $h(\Pe) = \tilde{g}(\Pe)$ and the continuity of $\tilde{g}$ is proved.

   From all of the above it follows that $\psi$ is a homeomorphism of $X_n$ onto $\phi_{\A}(\Pm(\I_n))$ for each $n$. The sequence $\{X_n\}$
   determines the topology of $X$ while the sequence $\{\phi_{\A}(\Pm(\I_n))\}$ determines the topology of $\Gl(\A)$ so we conclude that indeed
   $\psi$ is a homeomorphism of $X$ onto $\Gl(\A)$.

\end{proof}

\section{Points of local compactness} \label{S:Points}

We now investigate conditions that ensure the existence of a compact neighbourhood for a point in the Glimm space of a separable $C^*$-algebra.
We also discuss below the presence of the Baire property in the Glimm space of a separable $C^*$-algebra.

A map $q: X\to Y$ is called biquotient at $y\in Y$ if every open cover of $q^{-1}(y)$ has a finite subfamily $\mathcal{V}$ such that $y\in
\mathrm{Int} (\cup \{q(U) \mid U\in \mathcal{V}\})$.

$(i) \Rightarrow (ii)$ and $ (iii) \Rightarrow (i)$ below are from \cite{AP}. The first is problem nr. 18, chapter VI, the second is problem nr.
21, chapter VI.

\begin{thm} \label{T:Top}

   Let $X$ be a second countable locally compact Hausdorff space and $Y$ a Hausdorff quotient of it with $q$ the quotient map. The following are
   equivalent for $y\in Y$:
   \begin{itemize}
      \item[(i)] the quotient map $q$ is biquotient at $y$;
      \item[(ii)] $y$ has a compact neighbourhood;
      \item[(iii)] $y$ has a countable basis of neighbourhoods.
   \end{itemize}

\end{thm}

\begin{proof}

   By \cite[Theorem 1]{M} $Y$ is a paracompact space.
   $(i)\Rightarrow (ii)$. For each $x\in q^{-1}(y)$ we choose a compact neighbourhood $C_x$ of $x$ in $X$ and let $U_x$ be its interior. Then
   $\{U_x \mid x\in q^{-1}(y)\}$ is an open cover of $q^{-1}(y)$ so there are $\{x_i\}_{i=1}^k\subseteq q^{-1}(y)$ such that $y\in \mathrm{Int}
   \cup_{i=1}^k q(U_{x_i})$. The compact set $\cup_{i=1}^k q(C_{x_i})\supseteq \cup_{i=1}^k q(U_{x_i})$ is a neighbourhood of $y$.

   $(ii)\Rightarrow (iii)$. Let $V$ be a compact neighbourhood of $y$. There exists an increasing sequence $\{X_n\}$ of compact subsets of $X$
   such that $X = \cup_n X_n$. For each $n$ the set $q(X_n)$ is compact and the map $h\to h\circ q$, $h\in C(q(X_n))$, is a linear isometry of
   $C(q(X_n))$ into the separable $C^*$-algebra $C(X_n)$. Thus, by the normality of Y, there are continuous functions
   $\{f_n^m \mid 1\leq m < \infty\}$ of $Y$ into $[0,1]$ that separate the points of $q(X_n)$. The double sequence $\{f_n^m\mid_V \mid 1\leq m,n
   < \infty\}$ yields a one-to-one mapping of $V$ into $[0,1]^{\aleph_0}$ that is a homeomorphism onto its image by the compactness of $V$. It follows that
   the compact neighbourhood $V$ of $y$ is second countable.

   $(iii)\Rightarrow (i)$. Consider an open cover $\mathcal{U}$ of the closed subset $q^{-1}(y)$ of $X$. Since $X$ as well as $q^{-1}(y)$ are Lindel\"{o}f there is
   no loss of generality in supposing that $\mathcal{U}$ is countable, say $\mathcal{U} = \{U_n\}_{n=1}^{\infty}$. We may also suppose
   $U_n\subseteq U_{n+1}$ and $U_n\cap q^{-1}(y)\neq \mset$ for every $n$. We fix a neighbourhood basis $\{V_i\}_{i=1}^{\infty}$ of $y$ with
   $V_{i+1}\subseteq V_i$ for $i\in \mathbb{N}$.

   We claim that there is $n_0\in \mathbb{N}$ such that $q(U_{n_0})\supseteq V_{n_0}$; from this will follow that $q$ is biquotient at $y$.
   Assuming the claim to be false, there exists $y_n\in V_n\setminus q(U_n)$ for every $n$. The set $Q := \{y_n \mid n\in \mathbb{N} \}$
   satisfies $\overline{Q} = \{y\}\cup Q$ and $y\notin Q$. The subset $P := q^{-1}(Q)$ of $X$ is not closed since $Q$ is not closed in $Y$. Pick
   $x\in \overline{P}\setminus P$; then $q(x)\notin Q$ but $q(x)\in \overline{q(P)} = \overline{Q}$. Thus $x\in q^{-1}(y)$ and there exists $m\in
   \mathbb{N}$ with $x\in U_m$. Set $P_m := \cup \{q^{-1}(y_n) \mid n\geq m\}$. From $x\in \overline{P}\setminus P$ we get $x\in
   \overline{P_m}$; hence $U_m\cap P_m\neq \mset$. Thus there exists $m_1\geq m$ such that $U_m\cap q^{-1}(y_{m_1})\neq \mset$. But then
   $U_{m_1}\cap q^{-1}(y_{m_1})\neq \mset$ too. This contradicts the fact that $y_n\notin q(U_n)$, $n\in \mathbb{N}$ and the above claim is established.

\end{proof}

\begin{rems}

   \begin{enumerate}

      \item The implication $(i) \Rightarrow (ii)$ from Theorem \ref{T:Top} is valid for every locally compact Hausdorff space $X$. The implication $(iii)
      \Rightarrow (i)$ requires $X$ to be only a Lindel\"{o}f locally compact Hausdorff space.
      \item The space of the rational numbers $\mathbf{Q}$ is the image of a countable
      discrete space by a continuous one to one function. It is also the quotient of a locally compact Hausdorff space, being a first countable space,
      hence a k-space.
      On the other hand, the above theorem tells us that $\mathbf{Q}$ cannot be the quotient of a second countable locally compact space: it
      satisfies the condition (iii) but not the condition (ii).

   \end{enumerate}

\end{rems}

Let $X$ and $Y$ be as in Theorem \ref{T:Top} and denote by $S$ the (possibly empty) set of all the points $y\in Y$ that have the properties
mentioned in the statement. Then $S$ is open and locally compact in the relative topology.

\begin{prop}

   Let $X$, $Y$, and $S$ be as above. Then $Y$ is a Baire space if and only if $S$ is dense.

\end{prop}

\begin{proof}

   If $S$ is dense in $Y$ then clearly $Y$ is a Baire  space. Suppose now that $Y$ is a Baire space and let $\mset\neq G\subseteq Y$ be open. With $q :
   X\to Y$ being the quotient map we have that the subset $q^{-1}(G)$ of $X$ is open, thus $\sigma$-compact in $X$. Let $\{K_n\}$ be an
   increasing sequence of compact sets such that $q^{-1}(G) = \cup K_n$. Then $G = \cup q(K_n)$ and each $q(K_n)$ is compact, hence closed in $Y$.
   The open set $G$ is a Baire space too so there exists $n_0$ such that $\mathrm{Int} q(K_{n_0})\neq \mset$. The open set $\mathrm{Int} q(K_{n_0})$ is locally compact in
   its relative topology therefore it is contained in $S$. We obtained $S\cap G\neq \mset$.

\end{proof}

A separable $C^*$-algebra whose Glimm space is a non locally compact Baire space is constructed in \cite[III.9.2]{DH}. At the other extreme,
there exists a separable (even AF) $C^*$-algebra whose Glimm space contains no point that has a compact neighbourhood. Indeed, the space
constructed by Arhangel'ski\u{i} and Franklin \cite{AF} that was mentioned in the Introduction is nowhere locally compact (hence nowhere first
countable by Theorem \ref{T:Top}) but is the Hausdorff quotient of a countable disjoint sum of converging sequences and their limits and hence
is the Glimm space of an AF algebra by Theorem \ref{T:Eq}. However, this space like any Hausdorff quotient of a second countable locally compact
Hausdorff space has the property that each of its points is a $G_{\delta}$ point. Indeed, if $Y$ is such a quotient of $X$ by the map $q$ and
$y\in Y$ then $Z := q^{-1}(y)$ is closed so the open set $X\setminus Z$ is $\sigma$-compact. Thus $Y\setminus \{y\} = q(X\setminus Z)$ is an
$F_{\sigma}$ set and $\{y\}$ is a $G_{\delta}$ set.

\bibliographystyle{amsplain}
\bibliography{}

\end{document}